\documentclass[11pt]{article}      
\usepackage{graphicx}

\usepackage{latexsym}
\usepackage[cp1250]{inputenc}
\usepackage[OT4]{fontenc}
\usepackage{amsfonts}
\usepackage{amsmath}
\usepackage{amssymb}

\newcommand{\R}{\mathbb{R}}
\newcommand{\C}{\mathbb{C}}

\newcommand{\bo}{{\bf 0}}

\newcommand{\rank}{\operatorname{rank}}
\newcommand{\sgn}{\operatorname{sgn}}

\newcommand{\Aa}{\mathcal{A}}

\newcommand{\inv}{^{-1}}

\newcommand{\codim}{\operatorname{codim}}
\newcommand{\corank}{\operatorname{corank}}

\newtheorem{theorem}{Theorem}
\newtheorem{definition}{Definition}
\newtheorem{lemma}{Lemma}
\newtheorem{prop}{Propositon}
\newtheorem{cor}{Corollary}

\newenvironment{ex}{\medskip \noindent {\bf Example.\ }}{\bigskip}
\newenvironment{proof}{\par\noindent \emph{Proof. }}{\hspace*{\fill}$\Box$\par\medskip}

\title{On polynomial mappings from the plane to the plane\thanks{%
Iwona~Krzy\.{z}anowska and Zbigniew~Szafraniec\\
University of Gda\'{n}sk,
              Institute of Mathematics \\
              80-952 Gda\'{n}sk, Wita Stwosza 57, Poland\\          
              Email: Iwona.Krzyzanowska@mat.ug.edu.pl\\
              Email: Zbigniew.Szafraniec@mat.ug.edu.pl\\ \\
2000 \emph{Mathematics Subject Classification}:  Primary 14P99; Secondary 58K05\\
\emph{Keywords}: Singularities, cusps\\\\\emph{Supported by National Science Centre, grant 6093/B/H03/2011/40}
}
}

\author{Iwona Krzy\.{z}anowska \and Zbigniew~Szafraniec}

\date{July 2012}

\begin{document}

\def\nothanksmarks{\def\thanks##1{\protect\footnotetext[0]{\kern-\bibindent##1}}}

\nothanksmarks

\maketitle
%
%\pagestyle{fancy}
%
%\lhead{\fancyplain{}{\textsc{\small I.~Krzy\.{z}anowska, Z.~Szafraniec}}}
%\rhead{\fancyplain{}{\emph{\small Polynomial mappings into a Stiefel manifold and  immersions }}}
%

\maketitle

\abstract{Let $f:\R^2\longrightarrow\R^2$ be a generic polynomial mapping. There are constructed quadratic forms whose signatures determine the number of positive and negative cusps of $f$.}

\section{Introduction}
Mappings between $2$--manifolds are a natural object of study
in the theory of singularities. Let $M,N$ be smooth surfaces,
and let $f:M\rightarrow N$ be smooth. Whitney \cite{WhitneyCusps} proved
that critical points of a generic $f$ are folds and cusps.

If $M,N$ are oriented and $p$ is a cusp point, we define $\mu(p)$
to be the local topological degree of the germ $f:(M,p)\rightarrow (N,f(p))$.

There are several results concerning relations between the topology of $M,N$
and the topology of the critical locus of $f$ (see \cite{levine}, \cite{thom}, \cite{WhitneyCusps}).
In particular, there are results concerning $\sum_p\mu(p)$, where $p$
runs through the set of cusp points (see \cite{fukudaishikawa}, \cite{quine}). 
Singularities of map germs of the plane into the plane were studied in
\cite{fukudaishikawa}, \cite{gaffneymond},  \cite{moyaballesteros},  \cite{rieger}.
For a recent account of the subject, and other related results, we refer the reader
to \cite{dutertrefukui},\cite{saeki}.

In this paper we investigate the number of cusps of 
one--generic polynomial mappings.
Let $f=(f_1,f_2):\R^2\rightarrow\R^2$ be a polynomial mapping.
Denote
$J=\frac{\partial(f_1,f_2)}{\partial(x,y)}$, $F_i=\frac{ \partial(J,f_i)}{\partial(x,y)}$, $i=1,2$.
Let $I'$ be the ideal in $\R[x,y]$ generated by 
$J$, $F_1$, $F_2$, and $\frac{ \partial(J,F_i)}{\partial(x,y)}$, $i=1,2$.
We shall show that $f$ is one--generic if $I'=\R[x,y]$ (see Proposition \ref{alg1}).

Let $I$ be the ideal in $\R[x,y]$ generated by $J$, $F_1$, $F_2$, and let $\Aa=\R[x,y]/I$.
(In the local case, ideals defined by the same three generators were introduced and investigated
in \cite{fukudaishikawa}, \cite{gaffneymond}.)
Let $u:\R^2\rightarrow\R$ be a polynomial. Denote $U=\{p\in\R^2\ |\ u(p)>0\}$.

Assume that $\dim_\R\Aa<\infty$. In Section \ref{twierdzenia} we construct four quadratic forms
$\Theta_1,\ldots,\Theta_4$ on $\Aa$. We shall prove that signatures of these forms determine
the number of positive and negative cusps in $\R^2$ and in $U$ (Theorems \ref{th2}, \ref{th3}).
These signatures  may be computed using computer algebra systems.
In Section \ref{przyklady} we present examples which  were calculated with the help of {\sc Singular}
\cite{greueletal}.

\section{Mappings of the plane into the plane}
In this section we present useful facts about mappings of the plane into the plane.
In particular we show that definitions of fold points and cusp points introduced in
\cite{golub} and in \cite{WhitneyCusps} coincide (see Theorem \ref{oMapie}). In exposition and notation,
we follow closely \cite{golub}

Let $M,N$ be smooth manifolds,  and $p\in M$. For smooth mappings $f,g:M\longrightarrow N$ with $f(p)=g(p)=q$, we say that $f$ 
has first order contact with $g$ at $p$ if $Df(p)=Dg(p)$ as mapping of $T_pM\longrightarrow T_qN$. By
 $J^1(M,N)_{(p,q)}$ we shall denote the set of equivalence classes of mappings with $f(p)=q$ having the same
first order contact at $p$. Let $$J^1(M,N)=\bigcup_{(p,q)\in M\times N} J^1(M,N)_{(p,q)}.$$
An element $\sigma\in J^1(M,N)$ is called a 1--jet from $M$ to $N$.

 Denote $\corank \sigma=\min(\dim M, \dim N)-\operatorname{rank} Df(p)$.
 Put $S_r=\{\sigma\in J^1(M,N)|\ \corank \sigma =r\}$. According to \cite[Theorem 5.4]{golub},
 $S_r$ is a submanifold of $J^1(M,N)$ with $\codim S_r=r(|\dim M-\dim N|+r)$.
Given a  smooth mapping $f:M\longrightarrow N$ there is a canonically defined mapping $j^1f:M\longrightarrow J^1(M,N)$. 
 Let  $S_r(f)=\{p\in M|\ \corank Df(p)=r\}=(j^1f)\inv (S_r)$.

\begin{definition}
We say that $f:M\longrightarrow N$ is one--generic if $j^1f$ intersects $S_r$ transversely (denoted by  $j^1f\pitchfork S_r$)  for all $r$.
\end{definition}

According to \cite[Theorem 4.4]{golub}, if $j^1f\pitchfork S_r$ then either $S_r(f)=\emptyset$ or $S_r(f)$
 is a submanifold of $M$, with $\codim S_r(f)=\codim S_r$.

Assume that $M=N=\R^2$. In that case $J^1(\R^2,\R^2)\simeq \R^2\times \R^2\times M(2,2)$, 
where $M(2,2)=\{\left[ a_{ij} \right]|\ 1\leq i,j\leq 2\}$ is the set of $2\times 2$--matrices.

Let $\phi=a_{11}a_{22}-a_{12}a_{21}:J^1(\R^2,\R^2)\longrightarrow\R$. It is easy to see that
 $S_0=\{\phi\neq 0\}$, $S_1=\{\phi=0,\ d\phi\neq 0\}$ and $S_2=\{\phi=0,\ d \phi= 0\}$. 
In particular $\phi$ is locally a submersion at points of $S_1$. Moreover $\phi\circ j^1f=J$, where $J$
 is the determinant of the Jacobian matrix $Df$, and $J\inv (0)=S_1(f)\cup S_2(f)$.

\begin{lemma}\label{oneGeneric}
A mapping $f:\R^2\longrightarrow \R^2$ is one--generic if and only if $d J(p)\neq 0 $ for all $p\in J\inv(0)$. If that is the case then $S_1(f)=J^{-1}(0)$.
\end{lemma}
\begin{proof}As $\codim S_2=4$, then $j^1f\pitchfork S_2$ if and only if $S_2(f)=\emptyset$, i.e. $Df(p)\neq 0$ for $p\in J\inv(0)$. By \cite[Lemma 4.3]{golub}, $j^1f\pitchfork S_1$ if and only if $\phi\circ j^1f=J$ is locally a submersion at every $p\in S_1(f)$, i.e. $ d J(p)\neq 0$ for $p\in S_1(f)$.

If $f$ is one--generic, then  $S_2(f)=\emptyset$. Hence $J\inv (0)=S_1(f)$ and $d J(p)\neq 0$ for $p\in J\inv(0)$. 

If  $d J(p)\neq 0$ for all $p\in J\inv(0)$, then  $Df(p)\neq 0$ for all $p\in\R^2$. Hence $S_2(f)=\emptyset$ and $j^1f\pitchfork S_1$. \end{proof}

Take a one--generic mapping $f=(f_1,f_2):\R^2\longrightarrow \R^2$.
 For $p\in S_1(f)$ one of the following two conditions can occur. \begin{equation}\label{warunek1}
T_pS_1(f)+ \ker Df(p)=\R^2 ,
\end{equation} 
\begin{equation}\label{warunek2}
T_pS_1(f)=\ker Df(p).
\end{equation}
If $p\in S_1(f)$ satisfies  (\ref{warunek1}), then $p$ is called a fold point. 

Put $\bo=(0,0)\in\R^2$. Notice that the space $T_pS_1(f)$ is spanned by a vector $(-\frac{\partial J}{\partial y}(p),\frac{\partial J}{\partial x}(p))$, so we get
\begin{lemma}\label{falda}
A point $p\in \R^2$ is a fold point if and only if $J(p)=0$ and $$Df(p)\cdot \left [
\begin{array}{r} 
-\frac{\partial J}{\partial y}(p)\\
\frac{\partial J}{\partial x}(p) 
\end{array} \right ]\neq \bo,$$ i.e. $p$ is a regular point of $f|_{S_1(f)}$.

So $p\in S_{1}(f)$ satisfies  condition (\ref{warunek2}) if and only if $J(p)=0$ and $$Df(p)\cdot \left [
\begin{array}{r}
-\frac{\partial J}{\partial y}(p)\\ 
\frac{\partial J}{\partial x}(p)
\end{array} \right ]=\bo.$$
\end{lemma}

Assume that  condition (\ref{warunek2}) holds at $p\in S_1(f)$.
 Take  a smooth function $k$  on a  neighbourhood $U$ of $p$ such that $k\equiv0$ on $S_1(f)\cap U$ and $dk(p)\neq 0$. 
(By Lemma \ref{oneGeneric}, $J$ satisfies both these conditions.)
Let $\xi$ be a nonvanishing vector field  along $S_1(f)$ such that  $\xi $ is in the kernel of $Df$
at each point of $S_1(f)\cap U$. Then $dk(\xi )$ is a function on $S_1(f)$ having a zero at $p$.
The order of this zero does not depend on the choice of $k$ and $\xi$ (see \cite[p.146]{golub}),
so it equals the order of $dJ(\xi)$ at $p$.
\begin{definition}
We say that $p$ is a simple cusp if $p$ is a simple zero of $dJ(\xi)$. 
\end{definition} 

Let $F=(F_1,F_2):\R^2\longrightarrow\R^2$ be given by $$
F=\left [ \begin{array}{cc}
F_1\\
F_2\\
\end{array} \right ]=
[Df]\cdot \left [
\begin{array}{r}
-\frac{\partial J}{\partial y}\\
 \frac{\partial J}{\partial x} \\ 
\end{array} \right ].$$ 
So
$$F_1=-\frac{\partial f_1}{\partial x}\frac{\partial J}{\partial y}+
\frac{\partial f_1}{\partial y}\frac{\partial J}{\partial x}=\frac{\partial(J,f_1)}{\partial(x,y)},$$ 
$$F_2=-\frac{\partial f_2}{\partial x}\frac{\partial J}{\partial y}+
\frac{\partial f_2}{\partial y}\frac{\partial J}{\partial x}=\frac{\partial(J,f_2)}{\partial(x,y)}.$$

According to Lemma \ref{falda}, $p\in S_1(f)$ is a fold point if and only if $F(p)\neq \bo$.
\begin{lemma}\label{cusp1}
A point $p\in S_1(f)$ is a simple cusp if and  only if $F(p)=\bo$ and $$[DF(p)]\cdot\left [
\begin{array}{r}
-\frac{\partial J}{\partial y}(p)\\
 \frac{\partial J}{\partial x}(p) \\
\end{array} \right ]\neq \bo,$$   
 i.e. if $J(p)=0$, $F_1(p)=0$, $F_2(p)=0$, and either
$\partial(J,F_1)/\partial(x,y)(p)\neq 0$ or $\partial(J,F_2)/\partial(x,y)(p)\neq 0$.
\end{lemma}
\begin{proof}
Put $\xi_i=\left(\frac{\partial f_i}{\partial y},-\frac{\partial f_i}{\partial x}\right)$ on $S_1(f)$.
We have  $dJ(\xi_i)=F_i$. 
Then both $dJ(\xi_i)(p)=0$ if and only if $F(p)=\bo$. 
If that is the case then
 $p$ is a simple zero of at least one $dJ(\xi_i)$ if and only if $p$ is a regular point of at least one ${F_i}_{|S_1(f)}$, i.e.
 $$\left(\frac{\partial J}{\partial x}\frac{\partial F_i}{\partial y}-\frac{\partial J}{\partial y}\frac{\partial F_i}{\partial x}\right)(p)\neq 0.$$
 The last condition holds if and only if $$[DF(p)]\cdot\left [
\begin{array}{r}
-\frac{\partial J}{\partial y}(p)\\
 \frac{\partial J}{\partial x}(p) \\
\end{array} \right ]\neq \bo.$$
Of course the fields $\xi_1,\ \xi_2 $ are linearly dependent along $S_1(f)$, and at least one 
does not vanish at $p$. Without loss of generality we may assume that $\xi_1(p)\neq 0$,
so that $\xi_2=s\cdot\xi_1$, where $s$ is a smooth function on $S_1(f)$. In particular, $dJ(\xi_2)=s\cdot dJ(\xi_1)$.
 A short computation shows that $$Df\,\cdot\, \xi_1\equiv 0 \ \mbox{ on }\ S_1(f),$$ 
so that   $\xi_1$ is in  the kernel of $Df$ along $S_1(f)$.
Of course, both  $dJ(\xi_i)(p)=0$ if and only if $dJ(\xi_1)(p)=0$, and in this case
  $p$ is a simple zero of at least one $dJ(\xi_i)$ if and 
only if $p$ is a simple zero of $dJ(\xi_1)$, i.e. $p$ is a simple cusp. \end{proof}

Recall that $J\inv(0)=S_1(f)$ is a smooth $1$--manifold, and $d J\neq 0$ on $S_1(f)$. 
Take $p\in J\inv(0)$. We can find a smooth parametrization $\psi:(\R,0)\longrightarrow (J\inv (0),p)$. 
There exists a smooth nowhere zero function $\rho$ such that  
$$\frac{d\psi}{dt}(t)=\rho(t)\left (-\frac{\partial J}{\partial y}(\psi(t)),\frac{\partial J}{\partial x}(\psi(t))\right ).$$
 It is easy to check that $$\frac{d(f\circ \psi)}{dt}(t)=\rho(t)F(\psi(t)),$$
$$ \frac{d^2(f\circ \psi)}{dt^2}(t)=\rho '(t)F(\psi(t))+\rho^2(t)[DF(\psi(t))]\left [
\begin{array}{r}
-\frac{\partial J}{\partial y}(\psi(t))\\
 \frac{\partial J}{\partial x}(\psi(t)) \\
\end{array} \right ].$$ So we get   

\begin{theorem}\label{oMapie}
Suppose that $f:\R^2\rightarrow\R^2$ is one--generic. Then,
\begin{itemize}
\item[(i)] a point $p\in\R^2$ is a fold  if and only if  $J(p)=0$ and$$\frac{d(f\circ \psi)}{dt}(0)\neq \bo,$$ 
\item[(ii)] a point $p\in\R^2$ is a simple cusp if and only if $J(p)=0$,
$$\frac{d(f\circ \psi)}{dt}(0)=\bo \ \mbox{ and }\ \frac{d^2(f\circ \psi)}{dt^2}(0)\neq \bo,$$
 i.e. if $p$ is a non--degenerate critical point of $f_{|S_{1}(f)}$. 
\end{itemize}
\end{theorem}

\begin{theorem}[{\cite[Theorem 16A]{WhitneyCusps}, \cite[Theorem 2.4]{golub}}]\label{Whitney} A point $p$ is a simple cusp 
 if and only if the germ $f:(\R^2,p)\rightarrow (\R^2,f(p))$ is differentiably equivalent to the germ 
 $(u,v)\mapsto (u,uv+v^3)$.
\end{theorem}

Denote by $\mu(p)$ the local topological degree of the germ $f:(\R^2,p)\rightarrow(\R^2,f(p))$. Hence we have
\begin{cor}
If $p$ is a cusp point of $f$ then $\mu(p)=\pm1$.
\end{cor}

\section{Degree at a cusp point}
In this section we shall show how to interpret the sign of $\det DF(p)$ at a cusp point $p$.
\begin{lemma}
A translation in the domain does not change the determinant of $DF(p)$.
\end{lemma}
\begin{proof}Take $p=(x_0,y_0)$ and the translation $T(x,y)=(x+x_0,y+y_0)$. 
The determinant of the Jacobian matrix associated with $g=f\circ T$ equals $J\circ T$. With $g$ we can also associate a mapping $G$ the same way as in Section 1, i.e.
$$G= [Dg]\cdot \left [
\begin{array}{r}
-\frac{\partial(J\circ T)}{\partial y}\\
 \frac{\partial (J\circ T)}{\partial x} \\
\end{array} \right ].$$ 
Now it is easy to see that $G=F\circ T$, so $\det DG(\bo)=\det DF(p).$
 \end{proof}

\noindent A translation in the target obviously does not change $\det DF(p)$.
 From now on we shall assume that  $f:\R^2\rightarrow\R^2$ has a cups point at the origin and $f(\bo)=\bo$.
\begin{lemma}
An orthogonal change of coordinates in the target does not change  the determinant of $DF(\bo)$.
\end{lemma}
\begin{proof}Take an orthogonal isomorphism  $L(x,y)=(ax-by,bx+ay)$, where $a^2+b^2=1$. 
Then $\bo$ is a cusp point of $L\circ f$, and the determinant of the Jacobian
matrix associated with this mapping equals $(a^2+b^2)J=J$. With $g=L\circ f$ we can also associate  $$G=\left[\begin{array}{c}
G_1\\G_2\end{array}\right]=
D(L\circ f)\cdot\left[\begin{array}{r}  -\frac{\partial J}{\partial y}\\ \frac{\partial J}{\partial x}   \end{array}\right]=
 \left [
\begin{array}{cc}
a&-b\\b&a
\end{array} \right ]\cdot\left [Df\right ]\cdot\left [
\begin{array}{r}
-\frac{\partial J}{\partial y}\\
 \frac{\partial J}{\partial x} \\
\end{array} \right ]$$
$$=\left [
\begin{array}{cc}
a&-b\\b&a
\end{array} \right ]\cdot\left[\begin{array}{c}
F_1\\F_2\end{array}\right]=\left [
\begin{array}{cc}
a&-b\\b&a
\end{array}\right ] \cdot F=L\circ F.$$ 
Now it is easy to see that $\det DG(\bo)=\det DF(\bo).$ \end{proof}

\begin{lemma}
An orthogonal change of coordinates in the domain does not change  the determinant of $DF(\bo)$.
\end{lemma}
\begin{proof}Take an orthogonal isomorphism $R(x,y)=(cx-dy,dx+cy)$, where $c^2+d^2=1$. 
Then $\bo$ is a cusp point of $f\circ R$ and the determinant 
of the Jacobian matrix associated with this mapping equals $J\circ R$. With $g=f\circ R$ we can also associate 
$$G=\left[\begin{array}{c}
G_1\\G_2\end{array}\right]=D(f\circ R)\, \cdot\, \left[ \begin{array}{r}-\frac{\partial \, (J\circ R)}{\partial y}\\ \frac{\partial\, (J\circ R)}{\partial x} \end{array} \right]$$

$$=\left [Df\circ R\right ]\cdot \left [ 
\begin{array}{cc}
c&-d\\d&c
\end{array}  \right ]\, \cdot\, \left [
\begin{array}{cc}
c&d\\-d&c
\end{array} \right ]\cdot\left [
\begin{array}{r}
-\frac{\partial J}{\partial y}\circ R\\
 \frac{\partial J}{\partial x}\circ R \\
\end{array} \right ]$$
$$=\left[\begin{array}{c}
F_1\circ R\\F_2\circ R\end{array}\right]=F\circ R$$
Now it is easy to see that $\det DG(\bo)=(c^2+d^2)\det DF(\bo)=\det DF(\bo)$. \end{proof}

\begin{lemma}\label{zerowePochodne}
Assume that $\rank Df(\bo)=1$, $d J(\bo)\neq 0$ and $F(\bo)=0$. 
Then after an orthogonal change of coordinates in the domain and in the target we may have 
$$\frac{\partial f_1}{\partial x}(\bo)=\frac{\partial f_2}{\partial x}(\bo)=
\frac{\partial f_2}{\partial y}(\bo)=0,\ \  \frac{\partial f_1}{\partial y}(\bo)\neq 0,$$
$$\frac{\partial J}{\partial x}(\bo)=0,\ \frac{\partial J}{\partial y}(\bo)\neq 0.$$
\end{lemma}
\begin{proof}There exist $a,b,c,d\in \R$ such that $a^2+b^2=1$, $c^2+d^2=1$ and 
$$b\left(\frac{\partial f_1}{\partial x},\frac{\partial f_1}{\partial y}\right)+a\left(\frac{\partial f_2}{\partial x},\frac{\partial f_2}{\partial y}\right)=\bo,
\ \ c\,\frac{\partial J}{\partial x}+d\,\frac{\partial J}{\partial y}=0$$
at the origin.
 Let us consider two orthogonal isomorphisms: $L(x,y)=(ax-by,bx+ay)$ and $R(x,y)=(cx-dy,dx+cy)$. 
Put $g=(g_1,g_2)=L\circ f\circ R$.
Then $\rank Dg(\bo)=\rank Df(\bo)=1$ and 
\begin{equation}\label{pochodneFunkcji}\frac{\partial g_2}{\partial x}(\bo)=0, \ \ \frac{\partial g_2}{\partial y}(\bo)=0.\end{equation} 
 Let $J'$ denote the determinant of the Jacobian matrix $Dg$. 
Of course $J'=(a^2+b^2)(c^2+d^2)J\circ R=J\circ R$,
so that $dJ'(\bo)\neq 0$. One may check that
 \begin{equation}\label{pochoneJakobianu}\frac{\partial J'}{\partial x}(\bo)=0,\ \mbox{so}\ \frac{\partial J'}{\partial y}(\bo)\neq 0.\end{equation}
Then $(J')\inv (0)$ is  locally a smooth $1$--manifold near $\bo$, and there is $\varphi:(\R,0)\longrightarrow(\R,0)$ such that $J'(t,\varphi (t))\equiv 0$. 
Hence $0\equiv \frac{d}{dt}J'(t,\varphi (t))=\frac{\partial J'}{\partial x}(t,\varphi(t))+\frac{\partial J'}{\partial y}(t,\varphi(t))\varphi '(t)$. By (\ref{pochoneJakobianu}), $\varphi '(0)=0$. 
Because $F(\bo)=0$, then $\bo$ is a critical point of both $f|S_1(f)$ and $g|S_1(g)$. So  $g(t,\varphi(t))$
has a critical point at 0. By (\ref{pochodneFunkcji})
$$\bo=\frac{d}{d t}\left [
\begin{array}{c}
g_1(t,\varphi (t))\\
g_2(t,\varphi (t))
\end{array} \right ]_{|_{t=0}}=\left [\begin{array}{c}
\frac{\partial g_1}{\partial x}(\bo)+\frac{\partial g_1}{\partial y}(\bo)\varphi '(0)\\
\frac{\partial g_2}{\partial x}(\bo)+\frac{\partial g_2}{\partial y}(\bo)\varphi '(0)
\end{array} \right ]=\left [\begin{array}{c}
\frac{\partial g_1}{\partial x}(\bo)\\
0
\end{array} \right ],$$ 
and so $\frac{\partial g_1}{\partial x}(\bo)=0$. Because $\rank Dg(\bo)=1$, so $\frac{\partial g_1}{\partial y}(\bo)\neq 0$.  \end{proof}

Assume that $\rank Df(\bo)=1$, $d J(\bo)\neq 0$ and $F(\bo)=\bo$. 
Choose coordinates satisfying Lemma \ref{zerowePochodne}.
 If that is the case then 
$\frac{\partial J}{\partial y}(\bo)=-\frac{\partial f_1}{\partial y}(\bo)\frac{\partial^2 f_2}{\partial x\partial y}(\bo)$,
hence 
\begin{equation}\frac{\partial^2 f_2}{\partial x\partial y}(\bo)\neq 0.\label{r1}   \end{equation}
Moreover
$0=\frac{\partial J}{\partial x}(\bo)=-\frac{\partial f_1}{\partial y}(\bo)\frac{\partial^2 f_2}{\partial x^2}(\bo)$, 
so that 
\begin{equation}\label{r3}\frac{\partial^2 f_2}{\partial x^2}(\bo)=0.\end{equation}
As above there is smooth $\varphi:(\R,0)\longrightarrow(\R,0)$ such that 
$J(t,\varphi(t))\equiv 0$ and $\varphi '(0)=0$. 
Then 
$$0=\frac{d^2 }{d t^2}J(t,\varphi (t))_{|t=0}=\frac{\partial^2J}{\partial x^2}(\bo)+\frac{\partial J}{\partial y}(\bo)\varphi ''(0),$$ 
and $\varphi ''(0)=-\frac{\partial^2J}{\partial x^2}(\bo)/\frac{\partial J}{\partial y}(\bo).$
As $\bo$ is a cusp point then, according to Theorem \ref{oMapie},  
$$\frac{d^2}{d t^2}\left [
\begin{array}{c}
f_1(t,\varphi (t))\\
f_2(t,\varphi (t))
\end{array} \right ]_{|t=0}\neq \bo.$$
It is easy to see that $\frac{d^2}{d t^2}f_2(t,\varphi (t))_{|t=0}=0$, so
 \begin{equation}\label{r2}  0\neq\frac{d^2}{d t^2}f_1(t,\varphi(t))_{|t=0}=\frac{\partial^2f_1}{\partial x^2}(\bo)+\frac{\partial f_1}{\partial y}(\bo)\varphi ''(0)=\end{equation}
$$=\frac{\partial^2f_1}{\partial x^2}(\bo)-\left(\frac{\partial f_1}{\partial y}(\bo)\frac{\partial^2J}{\partial x^2}(\bo)\right)/\frac{\partial J}{\partial y}(\bo).$$
Two non--zero vectors 
$$v_1=\frac{d^2}{d t^2}\left [
\begin{array}{c} 
f_1(t,\varphi (t))\\
f_2(t,\varphi (t))
\end{array} \right ]_{|t=0}=\left [
\begin{array}{c}
\frac{\partial^2f_1}{\partial x^2}(\bo)-\frac{\partial f_1}{\partial y}(\bo)\frac{\partial^2J}{\partial x^2}(\bo)/\frac{\partial J}{\partial y}(\bo)\\
0
\end{array} \right ],$$
$$v_2=Df(\bo)\cdot \left[ \begin{array}{c}\frac{\partial J}{\partial x}(\bo)\\ \frac{\partial J}{\partial y}(\bo)      \end{array}     \right]=
\left [
\begin{array}{c}
\frac{\partial f_1}{\partial y}(\bo)\frac{\partial J}{\partial y}(\bo)\\
0
\end{array} \right ]$$
point in the same direction if and only if $$\frac{\partial f_1}{\partial y}(\bo)\left(\frac{\partial^2f_1}{\partial x^2}(\bo)\frac{\partial J}{\partial y}(\bo)-\frac{\partial f_1}{\partial y}(\bo)\frac{\partial^2J}{\partial x^2}(\bo)\right)>0.$$
\begin{lemma}\label{k2}
We have\begin{itemize}
\item[(i)] $\det DF(\bo)\neq 0$,
\item[(ii)] vectors $v_1$ and $v_2$ point in the same (resp. opposite) direction if and only if $\det DF(\bo)<0$ (resp. $\det DF(\bo)>0$).
\end{itemize}
\end{lemma}
\begin{proof}
We have $$F=(F_1,F_2)=\left(-\frac{\partial f_1}{\partial x}\frac{\partial J}{\partial y}+
\frac{\partial f_1}{\partial y}\frac{\partial J}{\partial x},-\frac{\partial f_2}{\partial x}\frac{\partial J}{\partial y}+
\frac{\partial f_2}{\partial y}\frac{\partial J}{\partial x}\right).$$ 
By (\ref{r1}),(\ref{r3}),(\ref{r2}),
 $$\det DF(\bo)=\det \left [
\begin{array}{cc}
-\frac{\partial^2 f_1}{\partial x^2}(\bo)\frac{\partial J}{\partial y}(\bo)+\frac{\partial f_1}{\partial y}(\bo)\frac{\partial^2 J}{\partial x^2}(\bo)&\frac{\partial F_1}{\partial y}(\bo) \\0&-\frac{\partial^2 f_2}{\partial x\partial y}(\bo)\frac{\partial J}{\partial y}(\bo)
\end{array} \right ]$$$$=\frac{\partial J}{\partial y}(\bo)\frac{\partial ^2f_2}{\partial x\partial y}(\bo)\left(\frac{\partial^2 f_1}{\partial x^2}(\bo)\frac{\partial J}{\partial y}(\bo)-\frac{\partial f_1}{\partial y}(\bo)\frac{\partial^2 J}{\partial x^2}(\bo)\right)\neq 0,$$

$$\sgn\det DF(\bo)=\sgn \bigg[-\frac{\partial f_1}{\partial y}(\bo)\left(\frac{\partial^2f_2}{\partial x\partial y}(\bo)\right)^2\left(\frac{\partial^2 f_1}{\partial x^2}(\bo)\frac{\partial J}{\partial y}(\bo)-\frac{\partial f_1}{\partial y}(\bo)\frac{\partial^2 J}{\partial x^2}(\bo)\right)\bigg]$$$$=(-1)\sgn \bigg[\frac{\partial f_1}{\partial y}(\bo)\left(\frac{\partial^2 f_1}{\partial x^2}(\bo)\frac{\partial J}{\partial y}(\bo)-\frac{\partial f_1}{\partial y}(\bo)\frac{\partial^2J}{\partial x^2}(\bo)\right)\bigg].$$

\end{proof}

\begin{lemma}\label{k1}
The local topological degree of the germ $f:(\R^2,\bo)\rightarrow(\R^2,\bo)$ equals $-1$ (resp. $+1$)
if vectors $v_1,v_2$ point in the same (resp. opposite) direction.
\end{lemma}
\begin{proof}
By Theorem \ref{Whitney} one may conclude that there exists  
open neighbourhoods $U,V$ of the origin in the plane such that
\begin{enumerate}
\item $f(U)=V$,
\item the set of regular values of $f:U\rightarrow V$, i.e. $V\setminus f(S_1(f)\cap U)$, consists of two connected components $V_1,V_2$
 such that
$q\in V_1$ if and only if $ f^{-1}(q)\cap U$ has one element, and
$q\in V_2$ if and only if $f^{-1}(q)\cap U$ has three element,
\item there exists a unit vector $v$ such that for any sequence $(q_n)\subset\bar V_2$ with $\lim q_n=\bo$ and $q_n\neq \bo$ we have
$$\lim\frac{q_n}{|q_n|}=v\ .$$
\end{enumerate}
Hence, if $q$ lies  close to the origin and the scalar product $\langle q,v\rangle$
is negative then $q\in V_1$.

The curve $(t,\varphi(t))$ is a smooth parametrization of $S_1(f)$ near
the origin. There is $\epsilon>0$ such that
 $\gamma(t)=f(t,\varphi(t))\in \bar{V}_2\setminus\{\bo\}$ for $0<|t|<\epsilon$.
By Theorem \ref{oMapie},
$$\frac{d\gamma}{dt}(0)=\bo\ ,\ v_1=\frac{d^2\gamma}{dt^2}(0)\neq\bo\ .$$
There exists smooth  $\alpha:\R,0\rightarrow\R^2$ such that
$\alpha(0)=v_1/2$ and $\gamma(t)=t^2\cdot\alpha(t)$. Then
$$v=\lim\frac{\gamma(t)}{|\gamma(t)|}=\lim\frac{\alpha(t)}{|\alpha(t)|}=\frac{v_1}{|v_1|},$$
so that vectors $v,v_1$ point in the same direction.

Let $\eta(t)=f(\frac{\partial J}{\partial x}(\bo)t,\frac{\partial J}{\partial y}(\bo)t)$. Then $\eta(0)=\bo$,
and
\[\frac{d\eta}{dt}(0)= 
Df(\bo)\cdot \left[\begin{array}{r}\frac{\partial J}{\partial x}(\bo)\\
\frac{\partial J}{\partial y}(\bo)    \end{array}\right]=v_2\neq\bo\ .
\]
There exists a smooth function $\beta:\R,0\rightarrow\R^2$ such that
$\beta(0)=v_2$ and $\eta(t)=t\cdot\beta(t)$.

If vectors $v,v_2$ point in the same direction then the scalar product
$\langle \eta(t),v\rangle $$=t\langle\beta(t),v\rangle$ is negative for all
$t<0$ sufficiently close to the origin. Then  $\eta(t)\in V_1$, so that 
$\eta(t)$ is a regular value and
$f^{-1}(\eta(t))\cap U=(\frac{\partial J}{\partial x}(\bo)t,\frac{\partial J}{\partial y}(\bo)t)$.

In this case  $J(\frac{\partial J}{\partial x}(\bo)t,\frac{\partial J}{\partial y}(\bo)t)$ is  negative,
hence the local topological degree of the mapping
$f:(\R^2,\bo)\rightarrow(\R^2,\bo)$ equals $-1$.

If vectors $v,v_2$ point in the oposite direction then the scalar product
$\langle \eta(t),v\rangle $$=t\langle\beta(t),v\rangle$ is negative for all
$t>0$ sufficiently close to the origin. As before,
$f^{-1}(\eta(t))\cap U=(\frac{\partial J}{\partial x}(\bo)t,\frac{\partial J}{\partial y}(\bo)t)$.

In that case  $J(\frac{\partial J}{\partial x}(\bo)t,\frac{\partial J}{\partial y}(\bo)t)$ is  positive,
hence the local topological degree of the mapping 
$f:(\R^2,\bo)\rightarrow(\R^2,\bo)$ equals $+1$. \end{proof}

\begin{prop}\label{th1}
Assume that $p$ is a cusp point of a mapping $f:\R^2\rightarrow\R^2$. Then $\det DF(p)\neq 0$, and the local
topological degree $\mu(p)$ of the germ $f:(\R^2,p)\rightarrow(\R^2,f(p))$
equals $\sgn\det DF(p)$.
\end{prop}
\begin{proof}A translation, as well as an orthogonal isomorphism,
does not change the local topological degree.
So we may assume that $p=f(p)=\bo$, and choose coordinates satisfying Lemma \ref{zerowePochodne}.
The assertion of the proposition is a consequence of Lemmas \ref{k2} and \ref{k1}. \end{proof}

\section{Polynomial mappings}\label{twierdzenia}
This section is devoted to the problem of determining the number of cusps
of a polynomial mapping
 $f=(f_1,f_2):\R^2\rightarrow\R^2$. Denote by $\Sigma$
the set of cusp points of $f$.
 Let $I$ be the ideal in $\R[x,y]$  generated by
$J$, $F_1$, $F_2$. Let  $I'$ be the one generated by
$J$, $F_1$, $F_2$, $\partial(J,F_1)/\partial(x,y)$, $\partial(J,F_2)/\partial(x,y)$.

\begin{prop}\label{alg1}
If $I'=\R[x,y]$ then $f$ is one-generic and the set of critical points
of $f$ consists of fold points and cusp points.
Moreover, $\Sigma=\{J=0,F_1=0,F_2=0\}$ is finite.
\end{prop}
\begin{proof}One may observe that $I'$ is contained in the ideal
generated by $J$, $\partial J/\partial x$, $\partial J/\partial y$. Therefore the last
ideal also equals $\R[x,y]$, and then its set of zeroes is empty.
Hence, if $J(p)=0$ then either $\partial J/\partial x(p)\neq 0$ or $\partial J/\partial y(p)\neq 0$.
By Lemma \ref{oneGeneric}, $f$ is one-generic.
 
Let $p$ be a critical point, so that $J(p)=0$.
Because the set of zeroes of $I'$ is empty, then either
$F_i(p)\neq 0$ or $\partial(J,F_i)/\partial(x,y)(p)\neq 0$ for some $i$.
By Lemma \ref{falda}, if $F_i(p)\neq 0$ then $p$ is a fold point.

If both $F_1(p)=0$, $F_2(p)=0$ then some $\partial(J,F_i)/\partial(x,y)(p)\neq 0$, and then
$p$ is a cusp point by Lemma \ref{cusp1}.
Thus $\Sigma$ is an algebraic set given by three equations $J=0$, $F_1=0$, $F_2=0$.
On the other hand $\Sigma$ is always discrete, and then  finite. \end{proof}

From now on we shall assume that $I'=\R[x,y]$.
By the  previous proposition, $\Sigma$ equals the set of zeroes of $I$.
Let $\Aa$ denote the $\R$--algebra $\R[x,y]/I$. Assume that $\dim_{\R}\Aa<\infty.$

For $h\in\Aa$, we denote by $T(h)$ the trace of the $\R$--linear endomorphism
$\Aa\ni a\mapsto h\cdot a\in\Aa$. Then $T:\Aa\rightarrow\R$ is a linear functional.
Take $\delta\in\R[x,y]$. Let $\Theta:\Aa\rightarrow\R$ be the quadratic form given by
$\Theta(a)=T(\delta\cdot a^2)$.

 According to \cite{becker}, \cite{pedersenetal}, the signature $\sigma(\Theta)$
of $\Theta$ equals
\begin{equation}\sigma(\Theta)=\sum \sgn\delta(p)\ ,\ \mbox{where}\ p\in\Sigma,\label{trace}\end{equation}
and if  $\Theta$ is non-degenerate then $\delta(p)\neq 0$ for each $p\in\Sigma$. 

Define quadratic forms
$\Theta_1(a)=T(a^2)$,
$\Theta_2(a)=T(\det DF\cdot a^2)$.

\begin{theorem}\label{th2}
Suppose that $I'=\R[x,y]$ and $\dim_R\Aa<\infty$. Then
\begin{itemize}
\item[(i)] $\#\Sigma=\sigma(\Theta_1),$
\item[(ii)] $\sum_{p\in\Sigma}\mu(p)=\sigma(\Theta_2)$,
\item[(iii)] $\#\{ p\in\Sigma\ |\ \mu(p)>0 \}=(\sigma(\Theta_1)+\sigma(\Theta_2))/2,$\\
$\#\{p\in\Sigma\ |\ \mu(p)<0\}=(\sigma(\Theta_1)-\sigma(\Theta_2))/2.$
\end{itemize}
\end{theorem}
\begin{proof}By Propositions \ref{th1}, \ref{alg1}, 
if $p$ is a zero of the ideal $I$ then $p\in\Sigma$ and $DF(p)\neq 0$.

Since $\Theta_1(a)=T(1\cdot a^2)$, by (\ref{trace}) its signature
equals $\sum_{p\in\Sigma}1=\#\Sigma$.
By (\ref{trace}) and Theorem \ref{th1},
the signature of $\Theta_2$ equals
$\sum_{p\in\Sigma}\sgn DF(p)=\sum_{p\in\Sigma}\mu(p)$.
Assertion (iii) is now obvious. \end{proof}
Take $u\in\R[x,y]$. Put $U=\{p\in\R^2\ |\ u(p)>0\}$.
The remainder of this section is devoted to the problem of determining
the number of cusps in $U$. 
Define  quadratic forms 
$\Theta_3(a)=T(u\cdot a^2)$,
$\Theta_4(a)=T(u\cdot \det DF\cdot a^2)$.

\begin{theorem}\label{th3}
Suppose that $I'=\R[x,y]$ and $\dim_R\Aa<\infty$. If $\Theta_3$ is non--degenerate then
\begin{itemize}
\item[(i)] $\Sigma\cap u^{-1}(0)=\emptyset$,
\item[(ii)] $\#\{p\in\Sigma\cap U\ |\ \mu(p)=+1  \}=(\sigma(\Theta_1)+\sigma(\Theta_2)+\sigma(\Theta_3)+\sigma(\Theta_4))/4$,
\item[(iii)] $\#\{p\in\Sigma\cap U\ |\ \mu(p)=-1  \}=(\sigma(\Theta_1)-\sigma(\Theta_2)+\sigma(\Theta_3)-\sigma(\Theta_4))/4$.
\end{itemize}
\end{theorem}
\begin{proof}As in the previous proof, $\det DF(p)\neq 0$ at each $p\in\Sigma$. Since $\Theta_3$ is non--degenerate, by (\ref{trace})
$u(p)\neq 0$ at each $p\in\Sigma$.

For $0\leq i,j\leq 1$ denote
$$a_{ij}=\#\{ p\in\Sigma\ |\ \sgn\det DF(p)=(-1)^i,\ \sgn u(p)=(-1)^j  \}.$$
These numbers satisfy the equations:
$$a_{00}+a_{10}+a_{01}+a_{11}=\sigma(\Theta_1),$$
$$a_{00}-a_{10}+a_{01}-a_{11}=\sigma(\Theta_2),$$
$$a_{00}+a_{10}-a_{01}-a_{11}=\sigma(\Theta_3),$$
$$a_{00}-a_{10}-a_{01}+a_{11}=\sigma(\Theta_4).$$
Now it is easy to verify that $a_{00}=(\sigma(\Theta_1)+\cdots+\sigma(\Theta_4)/4$,
$a_{10}=(\sigma(\Theta_1)-\sigma(\Theta_2)+\sigma(\Theta_3)-\sigma(\Theta_4))/4$.
\end{proof}
%%%%%%%%%%%%%%%%%%%%%%%%%%%%%%%%%%%%%%%%%%%%%%%%%%%%%%%%%%%%%%%%%%%%%%%%%%%%%%%%%%%%%%%%%%%%%%%
\section{Examples}\label{przyklady}
\begin{ex} Let $f=(f_1,f_2)=(xy^2-x^2+y^2+x-y, x-y):\R^2\longrightarrow\R^2$. It is easy to check that $$J=-2xy-y^2+2x-2y,\ F_1=-2xy^2+2y^3-4x^2-2y^2-2x+8y,\ F_2=2x+4y.$$ Using  {\sc Singular} one may verify that $I'=\R[x,y]$. According to Proposition \ref{alg1} the mapping $f$ is one-generic having only folds and cusps as critical points. Moreover the set of cusps $\Sigma$ is finite. The algebra $\Aa=\R[x,y]/I$ is two--dimensional, and has a basis $e_1=y$, $e_2=1$. Put $u=1-x^2-y^2$. The matrices of quadratic forms $\Theta_1,\Theta_2,\Theta_3,\Theta_4$ are $$\left[ \begin{array}{cc}+4&+2\\+2&+2
\end{array}\right ], \left[\begin{array}{cc}-96&-48\\-48&-48
\end{array}\right], \left[\begin{array}{cc}-76&-38\\-38&-18
\end{array}\right],24\cdot\left[\begin{array}{cc}+76&+38\\+38&+18
\end{array}\right].$$
So the quadratic form $\Theta_3$ is non--degenerate and $\sigma(\Theta_1)=2$, $\sigma(\Theta_2)=-2$, $\sigma(\Theta_3)=\sigma(\Theta_4)=0$. According to Theorems \ref{th2} and \ref{th3} the mapping $f$ has two cusps, both of negative sign, one of them lies in $U=\{u>0\}$.

\end{ex}

\begin{ex}
Put $f=(x^2y^3-x^2y+xy^2-x,x^3y-x^2y+y^3+x-y):\R^2\longrightarrow\R^2$ and $u=x^2+y^2-1$. Using the same method as before with the help of {\sc Singular}, one can check that $f$ is one-generic and the dimension of  $\Aa=\R[x,y]/I$ equals $38$. Moreover $f$ has eight cusps, six of them are positive and two are negative. All negative and three positive ones lie in $U=\{u>0\}$.
\end{ex}

\begin{ex}
Let $f=(10x^2y^3+4x^2y^2-2xy^3-6x^2y+8xy^2-5xy,5x^4y+10x^4-y^4+5x^2-3xy-9y)$ and $u=x-1$. In this case  $f$ is one-generic and the dimension of  $\Aa=\R[x,y]/I$ equals $56$. Moreover $f$ has six cusps, five of them are positive and one is negative. The negative one lies in $U=\{u>0\}$.
\end{ex}

\vspace{3cm}

%% \begin{figure}
%% \begin{center}
%% \includegraphics{figure.eps}
%% \caption{Caption}
%% \end{center}
%% \label{area}
%% \end{figure}

%% \begin{thebibliography}{9, 99 or Abc99}
%% \begin{thebibliography}{9}  for 1-digit labels
%% \begin{thebibliography}{99}  for 2-digit labels
%% \begin{thebibliography}{Abc}  for alphanumeric labels

\end{document}